\documentclass[12pt,reqno]{amsart}
\usepackage{times}
\usepackage{amsmath,amsfonts, amstext,amssymb,amsbsy,amsopn,amsthm}
\usepackage{dsfont}
\usepackage{esint}
\usepackage{graphicx}   
\usepackage{hyperref}
\usepackage[all]{xy}
\newcommand{\dR}{\mathds{R}}

\newcommand{\dZ}{\mathds{Z}}

\DeclareMathOperator{\PV}{P.V.}

\newcommand{\lap}{\mbox{$\bigtriangleup$}}

\newcommand{\ra}{{\mbox{$\rightarrow$}}}
\newcommand{\be}{\begin{equation}}
\newcommand{\ee}{\end{equation}}
\newtheorem{mthm}{Theorem}
\newtheorem{theorem}{Theorem}[section]
\newtheorem{lemma}[theorem]{Lemma}
\newtheorem{proposition}[theorem]{Proposition}
\newtheorem{corollary}[theorem]{Corollary}
\theoremstyle{definition}

\theoremstyle{remark}
\newtheorem{remark}{Remark}[section]

\theoremstyle{remark}

\numberwithin{equation}{section}

\begin{document}

\title{A direct method of moving spheres on fractional order equations}

\author{Wenxiong Chen }
\address{Department of Mathematics, Yeshiva University, New York, NY, USA}
\email{wchen@yu.edu}

\author{Yan Li}
\address{Department of Mathematics, Yeshiva University, New York, NY, USA}
\email{yali3@mail.yu.edu}

\author{Ruobing Zhang}
\address{Department of Mathematics, Princeton University, Princeton, NJ, USA}
\email{ruobingz@math.princeton.edu}

\date{\today}

\begin{abstract}
In this paper, we introduce a direct method of moving spheres for the nonlocal fractional Laplacian $(-\triangle)^{\alpha/2}$ with $0<\alpha<2$, in which a key ingredient is the narrow region maximum principle. As immediate applications, we classify the non-negative solutions for a semilinear equation involving the fractional Laplacian in $\mathbb{R}^n$; we prove a non-existence result for prescribing $Q_{\alpha}$ curvature equation on $\mathbb{S}^n$; then by combining the direct method of moving planes and moving spheres, we establish a Liouville type theorem on a half Euclidean space. We expect to see more applications of this method to many other nonlinear equations involving non-local operators.
\end{abstract}

\maketitle

\tableofcontents

\section{Introduction}
\label{s:introduction}

Recently, the fractional Laplacian has seen more and more applications in Physics, Chemistry, Biology, Probability, and Finance; and it has drawn more and more attention from the mathematical community. This fractional Laplacian is a pseudo-differential operator defined by
\begin{eqnarray}
(-\triangle)^{\alpha/2}u (x)&\equiv& C_{n,\alpha}\PV\int_{\dR^n}\frac{u(x)-u(y)}{|x-y|^{n+\alpha}}dy\nonumber\\
&\equiv&C_{n,\alpha}\lim\limits_{\epsilon\to0}\int_{\dR^n\setminus B_{\epsilon}(x)}\frac{u(x)-u(y)}{|x-y|^{n+\alpha}}dy,
\label{1.1}
\end{eqnarray}
for any real number $0<\alpha<2$.

Let
\begin{equation}
L_{\alpha}\equiv\Big\{u:\dR^n\to\dR^1\Big|\int_{\dR^n}\frac{|u(x)|}{1+|x|^{n+\alpha}}dx<\infty\Big\}.
\end{equation}
Then the operator $(-\triangle)^{\alpha/2}$ is well defined on the functions $u$ in $L_{\alpha} \cap C^{1,1}_{loc}$. One can see from the definition (\ref{1.1}) that it is nonlocal: Even $u$ is identically zero in a neighbourhood of a point $x$, $(-\triangle)^{\alpha/2}u(x)$ still may not vanish.
Hence, traditional methods on local differential operators, such as on Laplacian $-\triangle$ may not work on this nonlocal operator. To circumvent this difficulty,
Caffarelli and Silvestre \cite{CS} introduced the {\em extension method} that reduced this nonlocal problem into a local one in higher dimensions. For a function $u:\mathbb{R}^n \ra \mathbb{R}$, let  $U:\mathbb{R}^n\times[0, \infty) \ra \mathbb{R}$ be its extension satisfying
$$\left\{\begin{array}{ll}
div(y^{1-\alpha} \nabla U)=0, & (x,y) \in \mathbb{R}^n\times[0, \infty),\\
U(x, 0) = u(x), & x \in \mathbb{R}^n.
\end{array}
\right.
$$
Then
$$(-\lap)^{\alpha/2}u (x) = - C_{n,\alpha} \displaystyle\lim_{y \ra 0^+}
y^{1-\alpha} \frac{\partial U}{\partial y},  \;\; x \in \mathbb{R}^n.$$

This {\em extension method } provides a powerful tool and leads to very active studies in equations involving the fractional Laplacian, and a series of fruitful results have been obtained (see \cite{BCPS} \cite{CZ}    and the references therein).

Another approach on such nonlocal problems is to study the corresponding integral equations as introduced in \cite{CLO} and \cite{CLO1}.

However, when working at the extended problems or the corresponding integral equations, sometimes one needs to impose extra conditions on the solutions, which would not be necessary if one considers the pseudo differential equation directly (see the Introduction in \cite{CLL} for more details). Moreover, for equations involving the uniform elliptic nonlocal operators
\begin{equation}
C_{n,\alpha} \, \lim_{\epsilon \ra 0} \int_{\mathbb{R}^n\setminus B_{\epsilon}(x)} \frac{a(x-z)((u(x)-u(z))}{|x-z|^{n+\alpha}} dz = f(x,u),
\label{a}
\end{equation}
where
$$ 0 < c_0 \leq a(y) \leq C_1 ;$$
and for equations containing fully nonlinear nonlocal operators, such as
\begin{equation}
F_{\alpha}(u(x)) \equiv C_{n,\alpha} \, \lim_{\epsilon \ra 0} \int_{\mathbb{R}^n\setminus B_{\epsilon}(x)} \frac{G(u(x)-u(z))}{|x-z|^{n+\alpha}} dz =f(x,u)
\label{F}
\end{equation}
(see \cite{CS1} for the introductions of these operators), so far as we know, there has neither been any {\em extension method} nor {\em integral equation method} that work for these kinds of operators. This motivates us to come up with direct approaches on general nonlocal operators.

In our previous paper \cite{CLL}, a direct method of moving planes for the fractional Laplacian has been introduced and has been applied to obtain symmetry, monotonicity, and non-existence of solutions for various semi-linear equations involving the fractional Laplacian. Moreover, this direct approach can be applied to study the qualitative properties of solutions to uniformly elliptic problem (\ref{a}) \cite{TF}
as well as to fully nonlinear problem (\ref{F}) \cite{CLLg}.

In this paper, we will introduce another direct method--the method of moving spheres on the fractional Laplacian, which is more convenient than the method of moving planes in applications in some contexts.
For instance, in \cite{ChLi} and \cite{ChLi2}, for the case $\alpha =2$, the authors applied the method of moving spheres to prove a non-existence result for the prescribing scalar curvature equation, and hence answered an open question posed by Kazdan \cite{Kaz} that whether the well-known Kazdan-Warner necessary condition
is also sufficient for the existence of a solution in the case the given curvature function is rotationally symmetric. Here in Section \ref{s:prescribing}, we will extend this non-existence result to all real values of $\alpha$ between $0$ and $2$ by applying the direct method of moving spheres in the fractional setting (see Theorem 3 and Remark 1.1 for more details).

Similar to the method of moving planes, the method of moving spheres is a continuous application
of maximum principles. Hence, in Section 2, we first obtain the key ingredient in carrying out this method--the {\it (Spherically) Narrow Region Maximum Principle}:

\begin{mthm} \label{narrow-region-principle0}Let $w\in L_{\alpha}\cap C_{loc}^{1,1}(\Omega)$ be lower semi-continuous on $\bar{\Omega}$. Let $x_0$ be any point in $\mathbb{R}^n$. For $x \in B_{\lambda}(x_0)$, denote
$$x^{\lambda}\equiv\frac{\lambda^2 (x-x_0)}{|x-x_0|^2}+x_0,$$
the inversion point of $x$ about the sphere $S_{\lambda}(x_0) \equiv \{ x \mid |x-x_0|= \lambda\}$.

Assume that  $c(x)$ is bounded from below in $\Omega$ and
\begin{equation}
\begin{cases}
(-\triangle)^{\alpha/2}w(x)+c(x)w(x)\geq 0,\ x\in\Omega\subset B_{\lambda}(x_0),\nonumber\\
w(x)\geq 0,\  x\in B_{\lambda}(x_0)\setminus\Omega,
\nonumber\\
w(x)=-(\frac{\lambda}{|x-x_0|})^{n-\alpha}w(x^{\lambda}),\ x\in B_{\lambda}(x_0),
\end{cases}
\end{equation}
Then there exists some sufficiently small $\delta>0$ such that if $\Omega\subset \{x\in\dR^n|\lambda-\delta<|x-x_0|<\lambda\}$ (a spherically narrow region), then we have
$$
w(x)\geq 0, \;\; \forall x \in \Omega.
$$
Furthermore, if $w(x)=0$ for some $x\in\Omega$, then $w(x)\equiv 0$
for all $x\in\dR^n$.
\end{mthm}

As applications, in Section 3, we classify non-negative solutions of the semi-linear elliptic equations
\begin{equation}
(-\triangle)^{\alpha/2}u(x)=g(u(x)), \ x\in\dR^n.\label{g-eq0}
\end{equation}

\begin{mthm}\label{classification-theorem0}
Assume that $g:\dR_+^1\rightarrow \dR_{+}^1\cup\{0\}$
is locally bounded and
$\frac{g(r)}{r^{p}}$ is non-increasing with $p\equiv\frac{n+\alpha}{n-\alpha}$.
If $u\in L_{\alpha}\cap C_{loc}^{1,1}(\dR^n)$ is a nonnegative solution to equation (\ref{g-eq0}),
then one of the following holds:

$(1)$ For some $C_0\geq0$, $u(x)\equiv C_0$ for every $x\in\dR^n$ and $g(C_0)=0$.

$(2)$ There exists $\beta_1, \beta_2>0$, $x_0\in\dR^n$ such that
\begin{equation}
u(x)=\frac{\beta_1}{(|x-x_0|^2+\beta_2^2)^{\frac{n-\alpha}{2}}},\ \forall x\in\dR^n,
\end{equation}
and $g(r)$ is a multiple of $r^{p}$
for every $r\in(0,\max\limits_{x\in\dR^n}u(x)]$.
\end{mthm}

In Section 4, we study prescribing $Q_{\alpha}$ curvature equations on Riemannian manifolds and obtain a non-existence result using the method of moving spheres.

\begin{mthm}\label{prescribing-Q0}
Let $(\mathbb{S}^n, g_1)$ be the round sphere with standard metric of dimension $n \geq 2$.  Assume that $Q(x)$ is continuous and rotationally symmetric on $\mathbb{S}^n$, monotone in the region where $Q>0$, and $Q\not\equiv C$. Then for every $0<\alpha\leq 2$, the prescribing $Q_{\alpha}$-curvature equation
\begin{equation}
P_{\alpha}(u)=Q(x)  u^{\frac{n+\alpha}{n-\alpha}}, \ \; x \in \mathbb{S}^n
\label{PQC}
\end{equation}
 does not admit any positive solution.
\end{mthm}

For the precise definition of $P_{\alpha}$, please see section 4. In particular, it is the conformal Laplacian when $\alpha =2$ and the Penitz operator for $\alpha =4$.

\begin{remark}
In \cite{JLX1}, Jin, Li, and Xiong obtained a necessary condition--a Kazdan-Warner type identity
for (\ref{PQC}) to have a positive solution (see Proposition A.1 there). In the case $Q$ is rotationally symmetric, the condition becomes that, in order (\ref{PQC}) to have a positive solution,  $Q$ must not be monotone. Our Theorem \ref{prescribing-Q0} actually provides a stronger necessary condition that
\begin{equation}
 Q \mbox{ must not be monotone in the region where it is positive.}
 \label{NC}
 \end{equation}
 The significant part of this stronger necessary condition is that it is almost the sufficient condition to guarantee the existence of a solution. Actually, in the case $\alpha =2$, it is proved that, besides (\ref{NC}), if further assume that $Q$ is non-degenerate, then problem (\ref{PQC}) possesses a positive solution for $n \geq 2$ (see  \cite{XY} for $n=2$ and \cite{ChLi1} for $n \geq 3$). We believe that, the same existence results can be established for all real values of $\alpha$ between $0$ and $2$.

\end{remark}

Finally, in Section 5, combining the method of moving spheres and of moving planes, we present an alternative proof for the non-existence of positive solutions for the semilinear equation in a half space
\begin{equation}
\begin{cases}
(-\triangle)^{\alpha/2}u(x)=u^p(x),\ x\in\dR_+^n,\ 1<p\leq\frac{n+\alpha}{n-\alpha}\\
u(x)=0, \ x\not\in\dR_+^n,
\end{cases}\label{upper-half-space0}
\end{equation}
in both subcritical and critical cases.

\begin{mthm}
If $u\in L_{\alpha} \cap C_{loc}^{1,1}(\dR^n_+)$ is a nonnegative solution to (\ref{upper-half-space0}), then $u\equiv0$ in $\dR^n_+$.
\end{mthm}

\section{A Narrow Region Maximum Principle and the Method of Moving Spheres}
\label{s:maximum-principle}

In this section, we give the technical preparations for developing the method of moving spheres for the non-local operator $(-\triangle)^{\alpha/2}$.
Given $0<\alpha<2$, $\lambda>0$ and $x_0\in\dR^n$, we define the Kelvin transform of a function $u:\dR^n\to\dR^1$ centered at $x_0$ as follows, \begin{equation}u_{\lambda}(x)\equiv(\frac{\lambda}{|x-x_0|})^{n-\alpha}u(x^{\lambda}),
\label{Kelvin-transform}\end{equation}
where $$x^{\lambda}\equiv\frac{\lambda^2 (x-x_0)}{|x-x_0|^2}+x_0$$
is the inversion of $x$ with respect to the sphere
$$S_{\lambda}(x_0) \equiv \{ x \mid |x-x_0| = \lambda \}.$$

Without loss of generality, here we take $x_0=0$.
We start with a simple maximum principle
for anti-symmetric functions.

\begin{proposition}[Simple Maximum Principle]\label{strong-max-principle}
Let $\Omega$ be an open subset of $B_{\lambda}(0^n)$. Assume that $w\in L_{\alpha}\cap C_{loc}^{1,1}(\Omega)$ and is lower semi-continuous on $\bar{\Omega}$. If
\begin{equation}
\begin{cases}
(-\triangle)^{\alpha/2}w(x)\geq 0,\quad x\in\Omega,\\
w(x)\geq 0,\quad x\in B_{\lambda}(0^n)\setminus\Omega,\\
w(x)=-w_{\lambda}(x),\quad x\in B_{\lambda}(0^n),
\end{cases}
\end{equation}
then $w(x)\geq 0$ for every $x\in\Omega$. Moreover, if $w(x)=0$ for some $x\in\Omega$, then $w(x)\equiv0$ for every $x\in\dR^n$.
\end{proposition}

\begin{proof}

We argue by contradiction. Suppose that there is some $x_0\in\Omega$ such that $w(x_0)\equiv\min\limits_{x\in\Omega}w(x)<0$.

Let $\tilde{w}(x)\equiv w(x)-w(x_0)$, then immediately the following holds,
\begin{equation}
\begin{cases}
(-\triangle)^{\alpha/2}\tilde{w}(x)=(-\triangle)^{\alpha/2}w(x)\geq 0,\ \forall x\in\Omega,\\
\tilde{w}(x_0)=0,\\
\tilde{w}(x)\geq 0, \ x\in B_{\lambda}(0^n).
\end{cases}
\end{equation}

By the anti-symmetry assumption $w(x)=-w_{\lambda}(x)$, it holds that
\begin{eqnarray}
\Big(\frac{\lambda}{|x|}\Big)^{n-\alpha}\tilde{w}(x^{\lambda})&=&\Big(\frac{\lambda}{|x|}\Big)^{n-\alpha}w(x^{\lambda})-\Big(\frac{\lambda}{|x|}\Big)^{n-\alpha}w(x_0)\nonumber\\
&=&-w(x)+w(x_0)-\Big[1+\Big(\frac{\lambda}{|x|}\Big)^{n-\alpha}\Big]w(x_0)\nonumber\\
&> &-\tilde{w}(x).\label{anti-symmetry-estimate}
\end{eqnarray}

By straightforward calculations,
\begin{eqnarray}
(-\triangle)^{\alpha/2}\tilde{w}(x_0)&=&\int_{\dR^n}\frac{\tilde{w}(x_0)-\tilde{w}(y)}{|x_0-y|^{n+\alpha}}dy\nonumber\\
&=&\int_{B_{\lambda}(0^n)}\frac{-\tilde{w}(y)}{|x_0-y|^{n+\alpha}}dy+\int_{\dR^n\setminus B_{\lambda}(0^n)}\frac{-\tilde{w}(y)}{|x_0-y|^{n+\alpha}}dy\nonumber\\
&\equiv& I_1+I_2.\label{fractional-laplacian-at-x_0}\end{eqnarray}

Let $y\equiv\frac{\lambda^2 z}{|z|^2}$, then by (\ref{anti-symmetry-estimate}), it holds that
\begin{eqnarray}
I_2&=&\int_{B_{\lambda}(0^n)}\frac{-\tilde{w}(\frac{\lambda^2 z}{|z|^2})}{\Big|x_0-\frac{\lambda^2 z}{|z|^2}\Big|^{n+\alpha}}\cdot\frac{\lambda^{2n}}{|z|^{2n}}dz\nonumber\\
&<&\int_{B_{\lambda}(0^n)}\frac{\tilde{w}(z)}{\Big|x_0-\frac{\lambda^2 z}{|z|^2}\Big|^{n+\alpha}}\cdot\Big(\frac{\lambda}{|z|}\Big)^{n+\alpha}dz.
\end{eqnarray}

By (\ref{fractional-laplacian-at-x_0}), we have that
\begin{equation}
(-\triangle)^{\alpha/2}\tilde{w}(x_0)<\int_{B_{\lambda}(0^n)}\Big(\frac{1}{\big|\frac{|z|x_0}{\lambda}-\frac{\lambda z}{|z|}\big|}-\frac{1}{|x_0-z|^{n+\alpha}}\Big)\tilde{w}(z)dz.
\end{equation}
Notice that, for $z\in B_{\lambda}(0^n)$,
\begin{eqnarray}
\Big|\frac{|z|x_0}{\lambda}-\frac{\lambda z}{|z|}\Big|^2-|x_0-z|^2=\frac{(|x_0|^2-\lambda^2)(|z|^2-\lambda^2)}{\lambda^2}> 0,
\end{eqnarray}
which implies that
\begin{equation}
(-\triangle)^{\alpha/2}\tilde{w}(x_0)<0.
\end{equation}
A contradiction.

Now we assume that $w(x_0)=0$. Then by the same calculations as above, we have
\begin{equation}
0\leq(-\triangle)^{\alpha/2}w(x_0)=\int_{B_{\lambda}(0^n)}\Big(\frac{1}{\big|\frac{|z|x_0}{\lambda}-\frac{\lambda z}{|z|}\big|^{n+\alpha}}-\frac{1}{|x_0-z|^{n+\alpha}}\Big)w(z)dz\leq 0.
\end{equation}
Therefore, $w(x)=0$ for almost every $x\in B_{\lambda}(0^n)$.
By the anti-symmetry assumption $w(x)=-w_{\lambda}(x)$, it holds that
$w(x)=0$, for almost every $x\in\dR^n$.

\end{proof}

The following {\it narrow region principle} is the key to carry out the method of moving spheres.

\begin{theorem}[Narrow Region Principle] \label{narrow-region-principle}Let $w\in L_{\alpha}\cap C_{loc}^{1,1}(\Omega)$ be lower semi-continuous on $\bar{\Omega}$. If $c(x)$ is bounded from below in $\Omega$ and
\begin{equation}
\begin{cases}
(-\triangle)^{\alpha/2}w(x)+c(x)w(x)\geq 0,\ x\in\Omega\subset B_{\lambda}(0^n),\nonumber\\
w(x)\geq 0,\  x\in B_{\lambda}(0^n)\setminus\Omega,
\nonumber\\
w(x)=- w_{\lambda}(x),\ x\in B_{\lambda}(0^n),
\end{cases}
\end{equation}
Then there exists some sufficiently small $\delta>0$ such that the following holds. If $$\Omega\subset A_{\lambda-\delta,\lambda}(0^n) \equiv\{x\in\dR^n|\lambda-\delta<|x|<\lambda\},$$
 then we have
\begin{equation}
\inf\limits_{x\in\Omega}w(x)\geq 0.
\end{equation}
Furthermore, if $w(x)=0$ for some $x\in\Omega$, then $w(x)=0$
for almost every $x\in\dR^n$.
\end{theorem}

\begin{proof}

We argue by contradiction. Suppose there exists some $x_0\in\Omega$ such that
\begin{equation}
w(x_0)=\min\limits_{x\in\Omega}w(x)<0.
\end{equation}
Let $\tilde{w}(x)=w(x)-w(x_0)$, then $\tilde{w}(x_0)=0$ and
\begin{equation}
(-\triangle)^{\alpha/2}\tilde{w}(x)=(-\triangle)^{\alpha/2}w(x).
\end{equation}
By the anti-symmetry condition of $w$, we have that
\begin{equation}
\Big(\frac{\lambda}{|x|}\Big)^{n-\alpha}\tilde{w}(x^{\lambda})=-\tilde{w}(x)-\Big[1+\Big(\frac{\lambda}{|x|}\Big)^{n-\alpha}\Big]w(x_0).\label{narrow-anti-symmetry}\end{equation}

By straightforward calculations,
\begin{eqnarray}
(-\triangle)^{\alpha/2}\tilde{w}(x_0)&=&\int_{\dR^n}\frac{\tilde{w}(x_0)-\tilde{w}(y)}{|x_0-y|^{n+\alpha}}dy\nonumber\\
&=&\int_{B_{\lambda}(0^n)}\frac{-\tilde{w}(y)}{|x_0-y|^{n+\alpha}}dy+\int_{\dR^n\setminus B_{\lambda}(0^n)}\frac{-\tilde{w}(y)}{|x_0-y|^{n+\alpha}}dy\nonumber\\
&\equiv& J_1+J_2.\label{narrow-fractional-laplacian-at-x_0}\end{eqnarray}
By (\ref{narrow-anti-symmetry}), we have that
\begin{eqnarray}
J_2&=&\int_{\dR^n\setminus B_{\lambda}(0^n)}\frac{-\tilde{w}(y)}{|x_0-y|^{n+\alpha}}dy\nonumber\\
&=&\int_{\dR^n\setminus B_{\lambda}(0^n)}\frac{\Big(\frac{\lambda}{|y|}\Big)^{n-\alpha}\tilde{w}(y^{\lambda})}{|x_0-y|^{n+\alpha}}dy+\int_{\dR^n\setminus B_{\lambda}(0^n)}\frac{\Big[1+\Big(\frac{\lambda}{|y|}\Big)^{n-\alpha}\Big]w(x_0)}{|x_0-y|^{n+\alpha}}dy\nonumber\\
&\equiv& J_{21}+J_{22}.
\end{eqnarray}
Let $y=\frac{\lambda^2 z}{|z|^2}$, then we have that
\begin{eqnarray}
J_1+J_{21}=\int_{B_{\lambda}(0^n)}\Big(\frac{1}{\big|\frac{|z|x_0}{\lambda}-\frac{\lambda z}{|z|}\big|^{n+\alpha}}-\frac{1}{|x_0-z|^{n+\alpha}}\Big)\tilde{w}(z)dz.
\end{eqnarray}
Since $\tilde{w}(z)\geq 0$ for every $z\in B_{\lambda}(0^n)$, so $J_1+J_{21}\leq 0$. Notice that $w(x_0)<0$, then it holds that\begin{equation}
(-\triangle)^{\alpha/2}\tilde{w}(x_0)=(J_1+J_{21})+J_{22}\leq J_{22}\leq w(x_0)\int_{\dR^n\setminus B_{\lambda}(0^n)}\frac{1}{|x_0-y|^{n+\alpha}}dy.
\end{equation}

Now we estimate the right hand side of the above inequality,\begin{eqnarray}
\int_{\dR^n\setminus B_{\lambda}(0^n)}\frac{1}{|x_0-y|^{n+\alpha}}dy&\geq &\int_{(\dR^n\setminus B_{\lambda}(0^n))\cap(B_1(x_0)\setminus B_{\delta}(x_0))}\frac{1}{|x_0-y|^{n+\alpha}}dy\nonumber\\
&\geq&\frac{1}{4}\int_{B_1(x_0)\setminus B_{\delta}(x_0)}\frac{1}{|x_0-y|^{n+\alpha}}dy\nonumber\\
&\geq &\frac{C}{\delta^{\alpha}}.
\end{eqnarray}
The contradiction arises when $\delta>0$ is chosen sufficiently small.
\end{proof}

\section{Classification of Solutions to Semilinear Fractional Equations in $\dR^n$}

In this section, we will apply the tools developed in Section \ref{s:maximum-principle} to classify the non-negative solutions to the following semi-linear elliptic equations.

\begin{equation}
(-\triangle)^{\alpha/2}u(x)=g(u(x)), \ x\in\dR^n.\label{g-eq}
\end{equation}
Applying the method of moving spheres, we obtain the following
classification result for the positive solutions.

\begin{theorem}\label{classification-theorem}
Let $g:\dR_+^1\rightarrow \dR_{+}^1\cup\{0\}$
be a locally bounded function such that
$\frac{g(r)}{r^{p}}$ is non-increasing with $p\equiv\frac{n+\alpha}{n-\alpha}$.

If $u\in L_{\alpha}(\dR^n)\cap C_{loc}^{1,1}(\dR^n)$ is a nonnegative solution to equation (\ref{g-eq}),
then one of the following holds:

$(1)$ For some $C_0\geq0$, $u(x)\equiv C_0$ for every $x\in\dR^n$ and $g(C_0)=0$.

$(2)$ There exists $\beta_1>0$, $\beta_2>0$, $x_0\in\dR^n$ such that
\begin{equation}
u(x)=\frac{\beta_1}{(|x-x_0|^2+\beta_2^2)^{\frac{n-\alpha}{2}}},\ \forall x\in\dR^n,
\end{equation}
and $g(r)$ is a multiple of $r^{p}$
for every $r\in(0,\max\limits_{x\in\dR^n}u(x)]$.
\end{theorem}

The main ingredient in the proof of Theorem \ref{classification-theorem} is to establish the following  key lemma by applying the method of moving spheres. Precisely,

\begin{lemma}\label{l:key-lemma}Let $u$ be a positive solution to equation (\ref{g-eq}) and let $u_{\lambda}$ be the Kelvin transform of $u$ in the sense of (\ref{Kelvin-transform}).
 Then exactly one of the following holds.

$(A)$ For every $x_0\in\dR^n$, for all $\lambda\in(0,+\infty)$, it holds that
$u_{\lambda}(x)\geq u(x)$, $\forall x\in B_{\lambda}(x_0)$.

$(B)$ For every $x_0\in\dR^n$, there exists $\lambda_0\in(0,+\infty)$ which depends on $x_0$ such that $u_{\lambda_0}(x)\equiv u(x)$ for every $x\in B_{\lambda_0}(x_0)$.

\end{lemma}

\begin{proof}[Proof of Lemma \ref{l:key-lemma}]
Denote $w_{\lambda}(x)\equiv u_{\lambda}(x)-u(x)$.

The Lemma will be proven through the following steps.

{\textit{Step 1.}} We will show that for every sufficiently small $\lambda>0$, it holds that
$w_{\lambda}(x)\geq 0$ for every $x\in B_{\lambda}(x_0)$.

To this end, let $\lambda>0$ and define
\begin{equation}
\Sigma_{\lambda}^-\equiv\{x\in B_{\lambda}(x_0)|w_{\lambda}(x)<0\}.
\end{equation}
We will show that in fact $\Sigma_{\lambda}^-=\emptyset$ for every sufficiently small $\lambda>0$, which finishes the proof of {\textit{Step 1}}.

Now we are proving that there exists $\delta_0>0$ such that if $\Sigma_{\lambda}^-\neq\emptyset$, then
$\lambda\geq \delta_0$. In fact,
by the definition of $w_{\lambda}$, it holds that  for every $x\in B_{\lambda}(x_0)$, $w_{\lambda}(x)=-\Big(\frac{\lambda}{|x-x_0|}\Big)^{n-\alpha}w(x^{\lambda})$ and
\begin{equation}
(-\triangle)^{\alpha/2}w_{\lambda}(x)=\frac{g\Big((\frac{|x-x_0|}{\lambda})^{n-\alpha}u_{\lambda}(x)\Big)}{\Big((\frac{|x-x_0|}{\lambda})^{n-\alpha}u_{\lambda}(x)\Big)^p}u_{\lambda}^p(x)-\frac{g(u(x))}{u^p(x)}u^p(x).
\end{equation}
Then for every $x\in\Sigma_{\lambda}^-$,
$(\frac{|x-x_0|}{\lambda})^{n-\alpha}u_{\lambda}(x)<u(x).$
By assumption, $g(r)/r^p$ is non-increasing, so
\begin{equation}
\frac{g\Big((\frac{|x-x_0|}{\lambda})^{n-\alpha}u_{\lambda}(x)\Big)}{\Big((\frac{|x-x_0|}{\lambda})^{n-\alpha}u_{\lambda}(x)\Big)^p}\geq \frac{g(u(x))}{u^p(x)}
\end{equation}
and thus
\begin{eqnarray}
(-\triangle)^{\alpha/2}w_{\lambda}(x)&\geq&\frac{g(u(x))}{u^p(x)}(u_{\lambda}^p(x)-u^p(x))\nonumber\\
&=& \varphi(x)\cdot w_{\lambda}(x),
\end{eqnarray}
where
$\varphi(x)\equiv p\cdot\frac{g(u(x))}{u^p(x)}\cdot\psi(x)$, for some  $u_{\lambda}^{p-1}(x)\leq \psi(x)\leq u^{p-1}(x).
$
Notice that $\varphi\geq 0$ and hence
\begin{equation}
|\varphi(x)|\leq p\cdot\frac{g(u(x))}{u^p(x)}\cdot u^{p-1}(x)=p\cdot \frac{g(u(x))}{u(x)}.\label{phi-comp}
\end{equation}
Let $M_{\lambda}\equiv\sup\limits_{x\in B_{\lambda}(x_0)}u(x)<+\infty$, $N_{\lambda}\equiv\sup\limits_{x\in B_{\lambda}(x_0)}u^{-1}(x)<+\infty$ and $G_{\lambda}\equiv\sup\limits_{t\in(0,M_{\lambda})}g(t)$, then inequality (\ref{phi-comp}) implies that
\begin{equation}
\|\varphi\|_{L^{\infty}(B_{\lambda}(x_0))}\leq p\cdot G_{\lambda}\cdot N_{\lambda}<+\infty.
\end{equation}
Therefore, applying Theorem \ref{narrow-region-principle} and the local continuity of $w_{\lambda}$, $\Sigma_{\lambda}^{-}=\emptyset$
when $\lambda<\delta_0$, where $\delta_0>0$ is the constant in the statement of Theorem \ref{narrow-region-principle}.
We have completed the proof of {\textit{Step 1}}.

In fact, {\textit{Step 1}} provides a starting point to carry out the method of moving spheres for any given center $x_0\in\dR^n$. Then we will continuously increase the radius $\lambda$ of the sphere $S_{\lambda}(x_0) \equiv \partial B_{\lambda}(x_0)$ such that the inequality
$$w_{\lambda}(x)\geq 0 \; \text{holds for every} \, x\in B_{\lambda}(x_0).$$
For a given center $x_0\in\dR^n$, the \textit{critical scale}  $\lambda_0\in(0,+\infty]$ is defined as follows,
\begin{equation}
\lambda_0\equiv\sup\{\lambda>0 \mid w_{\mu}(x)\geq 0,\ \forall \ x\in B_{\mu}(x_0), \ \forall 0<\mu\leq\lambda\}.\label{critical-scale}
\end{equation}
Immediately, by definition, it holds that
\begin{equation}
w_{\lambda_0}(x)\geq 0,\ \forall x\in B_{\lambda_0}(x_0).\label{sign-at-lambda_0}
\end{equation}

The Lemma follows from  the analysis of the {\textit{global finiteness}} of the critical scale.
More precisely,

{\textit{Step 2.}}
Given $x_0\in\dR^n$, let the critical scale $\lambda_0>0$ be defined in (\ref{critical-scale}), then exactly one of the following holds:

$(a)$ For every $x_0\in\dR^n$, the corresponding critical scale $\lambda_0>0$ is finite, or

$(b)$  for every $x_0\in\dR^n$, the corresponding critical scale $\lambda_0=+\infty$, that is, for every $\lambda>0$
\begin{equation}
w_{\lambda}(x)\geq 0, \forall x\in B_{\lambda}(x_0).
\end{equation}

The statement of {\textit{Step 2}} can be reduced  to the following: if there exists $z_0\in\dR^n$ such that the corresponding critical scale $\mu_0 <\infty$, then for every $x_0\in\dR^n$ the corresponding critical scale $\lambda_0 < \infty$.

First, we need the following claim.

\vspace{0.5cm}
{\bf{Claim.}}
 If $\mu_0 < \infty$, then $u_{\mu_0}(x)=u(x)$ for every $x\in B_{\mu_0}(z_0)$.
\vspace{0.5cm}

We will prove the above claim by ruling out the following case: there exists $x'\in B_{\mu_0}(z_0)$
such that $w_{\mu_0}(x')\neq 0$.

First, by the definition of critical scale, \begin{equation}w_{\mu_0}(x)\geq 0,\ \forall x\in B_{\mu_0}(z_0).\end{equation}
Hence, we just assume that \begin{equation}w_{\mu_0}(x')>0.\label{1-point-positive}\end{equation}
Under this assumption, the first stage is to show that $w_{\mu_0}(x)>0$ for every $x\in B_{\mu_0}(x)$.
If the strict positivity of $w_{\mu_0}$ does not hold, we choose $x''\in B_{\mu_0}(z_0)$
such that $0=w_{\mu_0}(x'')=\min\limits_{x\in B_{\mu_0}(z_0)} w_{\mu_0}(x)$. That is, $u_{\mu_0}(x'')=u(x'')$.
By the calculations in the proof of {\textit{Step 1}},
\begin{eqnarray}
(-\triangle)^{\alpha/2}w_{\mu_0}(x'')&=&\frac{g\Big((\frac{|x''-z_0|}{\mu_0})^{n-\alpha}u_{\mu_0}(x'')\Big)}{\Big((\frac{|x''-z_0|}{\mu_0})^{n-\alpha}u_{\mu_0}(x'')\Big)^p}u_{\mu_0}^p(x'')-\frac{g(u(x''))}{u^p(x'')}u^p(x'')\nonumber\\
&=&\Bigg(\frac{g\Big((\frac{|x''-z_0|}{\mu_0})^{n-\alpha}u(x'')\Big)}{\Big((\frac{|x''-z_0|}{\mu_0})^{n-\alpha}u(x'')\Big)^p}-\frac{g(u(x''))}{u^p(x'')}\Bigg)u^p(x'')\nonumber\\
&\geq &0,
\end{eqnarray}
where the last inequality is due to the monotonicity of $g(t)/t^p$. Hence,
by Proposition \ref{strong-max-principle},
\begin{equation}
w_{\mu_0}(x)=0,\ \forall x\in B_{\mu_0}(z_0),
\end{equation}
which contradicts to the assumption. So the assumption (\ref{1-point-positive}) implies that
\begin{equation}
w_{\mu_0}(x)>0,\ \forall x\in B_{\mu_0}(z_0).\label{strict-positivity}
\end{equation}

The next stage is to rule out (\ref{strict-positivity}), which gives the proof of the claim. Let \begin{equation}0<m_0\equiv\min\limits_{x\in B_{\mu_0}(z_0)}w_{\mu_0}(x),\end{equation}
then by the continuity of $w_{\lambda}$ with respect to $\lambda$, there is some $0<\delta_1<\delta_0/10^{6}$ with $\delta_0$ in Theorem \ref{narrow-region-principle}
such that for every $\lambda\in[\mu_0,\mu_0+\delta_1)$, it holds that
\begin{equation}
w_{\lambda}(x)\geq m_0/2>0,\ \forall x\in B_{\mu_0}(z_0).\label{positivity-by-continuity}
\end{equation}
We take use of the same computations in {\textit{Step 1}} and by applying Theorem \ref{narrow-region-principle} and the local boundedness of $g$, then we obtain that
\begin{equation}
w_{\lambda}(x)>0,\ \forall x\in B_{\lambda}(z_0).\label{positivity-by-narrow}
\end{equation}
Notice that inequality (\ref{positivity-by-narrow}) contradicts to the definition of $\mu_0$.
Now we rule out (\ref{strict-positivity}), and thus we have proved the claim.

\vspace{0.5cm}

With  the above claim, we can finish
the proof of the reduced statement of {\textit{Step 2}}. We argue by contradiction and suppose the conclusion in the statement fails. That is, there exists $x_0\in\dR^n$ and a sequence $\{\lambda_j\}_{j\in\dZ_+}$ such that $\lim\limits_{j\to+\infty}\lambda_j=+\infty$ and
\begin{equation}
u_{\lambda_j}(x)\geq u(x),\ \forall x\in B_{\lambda_j}(x_0).\label{lambda_j-comp}
\end{equation}
Notice that, by the above claim,
\begin{equation}
\Big(\frac{\mu_0}{|x-z_0|}\Big)^{n-\alpha}\cdot u\Big(\frac{\mu_0^2(x-z_0)}{|x-z_0|^2}+z_0\Big)=u(x),\ \forall x\in B_{\mu_0}(z_0).
\end{equation}
Let $y\equiv\frac{\mu_0^2(x-z_0)}{|x-z_0|^2}+z_0$, then
\begin{eqnarray}
\lim\limits_{|y|\to+\infty}|y|^{n-\alpha}\cdot u(y)&=&\mu_0^{n-\alpha}\cdot\lim\limits_{|x|\to0}\Bigg(\frac{|y|}{\frac{\mu_0^2}{|x-z_0|}}\Bigg)^{n-\alpha}\cdot\Big(\frac{\mu_0}{|x-z_0|}\Big)^{n-\alpha}\cdot u(y)\nonumber\\
&=&\mu_0^{n-\alpha}u(z_0)\equiv D_0>0.\label{decay-order}
\end{eqnarray}
On the other hand, (\ref{lambda_j-comp}) implies that
\begin{equation}
\Big(\frac{\lambda_j}{|x_1-x_0|}\Big)^{n-\alpha}\cdot u\Big(\frac{\lambda_j^2(x_1-x_0)}{|x_1-x_0|^2}+x_0\Big)\geq u(x_1),\label{x_1-lambda_j}
\end{equation}
where $x_1\in\dR^n$
is chosen such that $x_1\in B_{\lambda_j}(x_0)$
for every sufficiently large $j$.
By (\ref{decay-order}) and taking the limit of (\ref{x_1-lambda_j}), it holds that
\begin{equation}
0=\lim\limits_{j\to+\infty}\frac{1}{\lambda_j^{n-\alpha}}\cdot D_0\geq u(x_1).
\end{equation}
The above inequality contradicts $u>0$ everywhere in $\dR^n$. This completes the proof of {\textit{Step 2}} and thus of Lemma \ref{l:key-lemma}.
\end{proof}

Now we proceed to prove Theorem \ref{classification-theorem}. The proof is a combination of Lemma \ref{l:key-lemma} and the following fact.

\begin{proposition}
[\cite{LZ}] \label{p:classification} Let $u\in C^{1,1}(\dR^n)$ and let $u_{\lambda}$ be its Kelvin transform,
in the notation of Lemma \ref{l:key-lemma}, then we have the following:

$(C) $ If $(A)$ holds, then $u$ is constant.

$(D) $ If $(B)$ holds, then there exists $\gamma>0$, $\beta>0$ such that
\begin{equation}
u(x)=\frac{\gamma}{(|x-x_0|^2+\beta^2)^{\frac{n-\alpha}{2}}},\ \forall x\in\dR^n.\label{bubble}
\end{equation}

\end{proposition}

We will use a modified version of $(C)$ in the proof of our later theorem. So for the convenience of readers, we give the proof of $(C)$.

\begin{proof}
[Proof of $(C)$ in proposition \ref{p:classification}]

It suffices to show that for every $z\in\dR^n$, $\nabla u(z)=0$. This actually is a result of \cite{LZ}, however, for reader's convenience, we outline it here. For $y\in\dR^n$, let
\begin{equation}
\lambda\equiv|z-y|,\ x\equiv\frac{|x-y|(z-y)}{\lambda}+y,
\end{equation}
and
\begin{equation}
t\equiv\frac{|x-y|}{\lambda}.
\end{equation}
Then consider the Kelvin transform about the center $y\in\dR^n$,
\begin{eqnarray}
w_{\lambda}(x)&=&u_{\lambda}(x)-u(x)\nonumber\\
&=&(\frac{1}{t})^{n-\alpha}u(\frac{z-y}{t}+y)-u(t(z-y)+y)\equiv h(t).
\end{eqnarray}
It is clear that $h(1)=0$
and $h(t)<0$
for every $t>1$, and thus
\begin{equation}
0\geq\frac{d}{dt}\Big|_{t=1}h(t)=-(n-\alpha)u(z)-2\langle\nabla u(z),z-y\rangle.
\end{equation}
Hence, for every $y\in\dR^n$  with $|z-y|\neq 0$, we have that
\begin{equation}
\frac{-(n-\alpha)u(z)-2\langle\nabla u(z),z-y\rangle}{|z-y|}\leq 0.
\end{equation}
Let $\nu\equiv\frac{z-y}{|z-y|}$, we have that $\langle\nabla u(z),\nu\rangle\geq 0$
by letting $|y|\rightarrow+\infty$.
Since $y\in\dR^n$ is arbitrary, $\nabla u(z)=0$. This completes the proof.

\end{proof}

With the above results, we can finish the proof Theorem \ref{classification-theorem}.

\begin{proof}
[Proof of Theorem \ref{classification-theorem}]

By Lemma \ref{l:key-lemma} and proposition \ref{p:classification},
if $u\in L_{\alpha}(\dR^n)\cap C^{1,1}(\dR^n)$ is a non-negative solution to (\ref{g-eq}), then either $(C)$ or $(D)$ holds. If $(C)$ holds,
there exists $C_0\geq 0$ such that $u\equiv C_0$ and thus $g(C_0)=(-\triangle)^{\alpha/2}C_0=0$. Now assume that  $(D)$ holds. Notice that (\ref{bubble}) is a solution to
\begin{equation}
(-\triangle)^{\alpha/2}u= g(u).
\end{equation}
The function $g$ satisfies that for some $\gamma_0> 0$,
\begin{equation}
g(r)=\gamma_0\cdot r^p,\ \forall r\in(0,\max\limits_{x\in\dR^n}u(x)].
\end{equation}
The completes the proof of the theorem.

\end{proof}

\section{A Non-Existence Result for Prescribing $Q_{\alpha}$ Curvature Equation on $\mathbb{S}^n$}
\label{s:prescribing}

In this section, we will prove a non-existence result for the prescribing fractional order $Q_{\alpha}$ curvature
equation on $\mathbb{S}^n$.

 A natural question in conformal geometry is the following:

 Let $(\mathbb{S}^n,g_1)$ be the round sphere of dimension $n\geq 3$ such that $\sec_{g_1}\equiv1$.
 Given $0<\alpha<2$ and a smooth function $Q\in C^{\infty}(\mathbb{S}^n)$, does there exist a conformal metric $\tilde{g}=u^{\frac{4}{n-\alpha}}g_1$ such that
$Q_{\alpha}(x)\equiv Q_{\alpha}[g](x)=Q(x)$ for every $ x\in\mathbb{S}^n$?

In this direction, one of the most important
points is to find a necessary and sufficient condition
on the function $Q$ such that the conformal metric $g$ exists under that condition.
 When $\alpha=2$, it turns out to be that $Q_{\alpha}$ is the scalar curvature up to a dimensional constant, and the corresponding prescribing scalar curvature problem is called the Nirenberg problem. The studies of such a problem leads to a very active research area in the last decade. The pioneer
work in the direction of prescribing $Q_{\alpha}$ curvature problem can be found in \cite{JLX1}
 and \cite{JLX2}.
Our main goal in this section is to extend a non-existence result in \cite{ChLi} to the fractional setting.

We will focus on the fractional GJMS operator $P_{\alpha}$ on a round sphere $(\mathbb{S}^n,g_1)$ with $\sec_{g_1}\equiv1$ in the sense of \cite{CG}.
For the convenience of readers, let us briefly introduce the preliminary materials.

\subsection{Fractional GJMS Operators and Prescribing Fractional Curvature Problem}
We will start with the definition of fractional GJMS operator.
Let $(\mathbb{H}^{n+1},g_{H})$ be the $(n+1)$-dimensional hyperbolic space with $\sec_{g_{H}}\equiv-1$. It is standard that the $n$-sphere $\mathbb{S}^n$ with the conformal structure $[g_1]$ can be viewed as the conformal infinity of $(\mathbb{H}^{n+1},g_{H})$. Specifically, the conformal compactification is given by the following coordinates,
\begin{equation}
g_H=\frac{dt^2+(1-\frac{t^2}{4})^2g_1}{t^2},\ t\in[0,2].
\end{equation}
Clearly, $(t^2g_H)|_{t=0}=g_1$.
By \cite{CG}, for any representative $h\in[g_1]$, the fractional GJMS operator $P_{\alpha}[h]\equiv P_{\alpha}[h,g_H]$ can be defined in terms of the scattering operator. Precisely, consider the
Poisson equation with $s\in(0,n)$,
\begin{equation}
-\triangle_{g_H} v-s(n-s)v=0,
\end{equation}
and the solution of the form
\begin{equation}
v=y^{n-s}F+y^sH,
\end{equation}
with $F,G\in C^{\infty}(\mathbb{H}^{n+1})$ and $y$ is a geodesic defining function of $(\mathbb{S}^n,h)$ such that
\begin{enumerate}
\item there is some $\epsilon>0$ such that $|\nabla y|_{y^2g_+}\equiv 1$ holds on $\mathbb{S}^n\times[0,\epsilon)$,

\item $y^2g_{H}|_{y=0}=h$.
\end{enumerate}
By standard calculations (see \cite{CG} or \cite{GZ}), we have the following expansion of $F$ and $G$ in terms of $y$,
\begin{equation}
F=f_0+f_2\cdot y^2+\ldots,
\end{equation}
and
 \begin{equation}
H=h_0+h_2\cdot y^2+\ldots.
\end{equation}
Then the scattering operator $S(s)$
can be defined as follows,
if $F|_{y=0}=f$, then $S(s)(f)\equiv h$
with $h=H|_{y=0}$.
Now for any $\alpha\in(0,n)$, the fractional GJMS operator $P_{\alpha}[h]$ is defined by
\begin{equation}
P_{\alpha}[h](f)\equiv2^{\alpha}\cdot\frac{\Gamma(\frac{\alpha}{2})}{\Gamma(-\frac{\alpha}{2})}\cdot S\Big(\frac{n+\alpha}{2}\Big)(f).
\end{equation}
In the above definition, the order $\alpha$ can be any real number in $(0,n)$. Moreover, the above procedure can be generalized to any conformally compact Einstein manifold $(X^{n+1},g_+)$ with a conformal infinity $(M^n,[h])$ (for more details, see \cite{CG}).
In general, the fractional GJMS operator
has the following conformal covariance property: if $\hat{h}=u^{\frac{4}{n-\alpha}}h$, then for every $v\in C^{\infty}(M^n)$,
\begin{equation}
P_{\alpha}[\hat{h},g_+](v)=u^{-\frac{n+\alpha}{n-\alpha}}P_{\alpha}[h,g_+](uv).\label{conformal-covariance}
\end{equation}
For $0<\alpha<n$, we define the $Q_{\alpha}$ curvature
\begin{equation}Q_{\alpha}\equiv\frac{n-\alpha}{2}\cdot P_{\alpha}[h,g_+](1)\end{equation}
 with respect to the metric $h$ on $M^n$.

In the rest of this section, we will briefly recall the prescribing $Q_{\alpha}$ curvature problem with $0<\alpha<2$. By \eqref{conformal-covariance}, the existence of the conformal metric $h=\tilde{u}^{\frac{4}{n-\alpha}}g_1$ with $Q_{\alpha}[h]=Q$ is equivalent to the existence of positive solution to
\begin{equation}
P_{\alpha}[g_1](\tilde{u})(x)=\frac{2}{n-\alpha}\cdot Q(x)\cdot \tilde{u}^{\frac{n+\alpha}{n-\alpha}}(x).\label{Q-eq-on-sphere}
\end{equation}
Next, we will see that the prescribing $Q_{\alpha}$ curvature problem on $\mathbb{S}^n$ can be reduced to the equation on Euclidean space via stereographic projection.
Let $\pi_S: \mathbb{S}^n\setminus\{S\}\to\dR^n$ be the stereographic projection, then the pullback metric $(\pi^{-1})^* g_1$ is conformal to $g_0$, the Euclidean metric on $\dR^n$. Indeed, for every $x\in\dR^n$,
\begin{equation}
(\pi_S^{-1})^* g_1\equiv\Big(\frac{2}{1+|x|^2}\Big)^2\cdot g_0.
\end{equation}
Therefore, $g_1$ can be viewed as a conformal metric of $\dR^n$. Let
$\eta(x)\equiv\Big(\frac{2}{1+|x|^2}\Big)^{\frac{n-\alpha}{2}}$, then
\begin{equation}
g_1=\eta^{\frac{4}{n-\alpha}}g_0.
\end{equation}
By the conformal covariance property \eqref{conformal-covariance}, it holds that
\begin{equation}
P_{\alpha}[g_1](\tilde{u})(x)=\eta^{-\frac{n+\alpha}{n-\alpha}}(-\triangle)^{\alpha/2}(\tilde{u}\cdot\eta)(x),\ \forall x\in\dR^n.
\end{equation}
Now let $\tilde{u}$ be a positive solution to (\ref{Q-eq-on-sphere}) and denote by $u(x)\equiv\tilde{u}(x)\cdot\eta(x)$ for every $x\in\dR^n$,
 then the above identity and equation (\ref{Q-eq-on-sphere}) imply that
\begin{equation}
(-\triangle)^{\alpha/2}u (x)=\frac{2}{n-\alpha}\cdot Q(x)\cdot u^{\frac{n+\alpha}{n-\alpha}}(x),\ \forall x\in\dR^n.\end{equation}
For convenience, we consider the following equation,
\begin{equation}
(-\triangle)^{\alpha/2}u (x)= Q(x)\cdot u^{p}(x),\ \forall x\in\dR^n,\label{prescribing-on-R^n}
\end{equation}
with $p\equiv\frac{n+\alpha}{n-\alpha}$.
By assumption, $\tilde{u}$
is smooth and thus bounded on $\mathbb{S}^n$, which implies that
\begin{equation}
|x|^{n-\alpha}u(x)=\tilde{u}(x)\cdot\Big(\frac{2|x|^2}{1+|x|^2}\Big)^{\frac{n-\alpha}{2}}\longrightarrow 2^{\frac{n-\alpha}{2}}\cdot\tilde{u}(S)\equiv M_0,\ \text{as}\ |x|\to\infty.\label{decay-condition}
\end{equation}
From now on, we just focus on the equation \eqref{prescribing-on-R^n} in the Euclidean space with decay condition (\ref{decay-condition}).

In the study of prescribing $Q_{\alpha}$ curvature problem, it is very natural to start with some smooth function $Q$
with certain symmetry. In this section, we assume that the smooth function $Q$
is rotationally symmetric. In \cite{ChLi}, the authors proved the following non-existence theorem for prescribing scalar curvature problem.
\begin{theorem}
[\cite{ChLi}]\label{t:Chen-Li-theorem} Let $R$ be smooth and rotationally symmetric. If $R$ is monotone in the region where $R>0$ and $R\not\equiv C$, then the prescribing scalar curvature problem on sphere does not admit any positive solution.
\end{theorem}

We will extend theorem \ref{t:Chen-Li-theorem} to the fractional setting.
First we reduce the corresponding assumption on the radial function $Q(x)\equiv Q(|x|)$ to the following
\begin{equation}
\begin{cases}
Q\in C^{\infty}(\dR^n),\\
Q(r)>0,\ Q'(r)\leq 0,\ \forall r<1,\\
 Q(r)\leq 0,\ \forall r\geq 1,
\end{cases}\label{Q-condition}
\end{equation}
where $r\equiv|x|$.
With the above preparations, we are ready to state and prove the non-existence result of the prescribing $Q_{\alpha}$ curvature problem in the following.

\subsection{A Non-Existence Theorem}
\label{ss:non-existence}

We will apply the method of moving spheres to prove the non-existence result of the prescribing $Q_{\alpha}$ curvature equation on the sphere.

\begin{theorem}\label{prescribing-Q}
Let $Q$ be continuous and rotationally symmetric on the round sphere $(\mathbb{S}^n,g_1)$ of curvature $\equiv 1$. In addition, assume that $Q$ is monotone in the region where $Q>0$ and $Q\not\equiv C$. Then for every $0<\alpha\leq 2$, the prescribing $Q$-curvature equation on $(\mathbb{S}^n,g_1)$
\begin{equation}
P_{\alpha}(\tilde{u})=Q\cdot \tilde{u}^{p},\ p\equiv\frac{n+\alpha}{n-\alpha},\label{prescribing-on-sphere}
\end{equation}
 does not admit any positive solution $\tilde{u}\in \mathcal{D}(P_{\alpha})\cap C^{1,1}(\mathbb{S}^n)$.
\end{theorem}

\begin{proof}

To prove the theorem, it suffice to show the non-existence of positive $C^{1,1}$ solution to the equation (\ref{prescribing-on-R^n}) with
(\ref{decay-condition}) and (\ref{Q-condition}).

Let
\begin{equation} x^{\lambda}\equiv\frac{\lambda^2 x}{|x|^2}
,\
u_{\lambda}(x)\equiv \Big(\frac{\lambda}{|x|}\Big)^{n-\alpha}u(x^{\lambda}),\end{equation}
and
\begin{equation}
w_{\lambda}(x)\equiv u_{\lambda}(x)-u(x).
\end{equation}
Straightforward computations give that
\begin{equation}
(-\triangle)^{\alpha/2}u_{\lambda}(x)=Q\Big(\frac{\lambda^2}{r}\Big)u_{\lambda}^p(x),
\end{equation}
and then
\begin{equation}
(-\triangle)^{\alpha/2}w_{\lambda}(x)=Q\Big(\frac{\lambda^2}{r}\Big)u_{\lambda}^p(x)-Q(r)u^p(x) \label{equation-of-w}\end{equation}
Since $Q(r)>0$, $Q(1/r)\leq 0$ for every $0<r<1$,
it holds that for $\lambda=1$,\begin{equation}
(-\triangle)^{\alpha/2}w_{1}(x)<Q(1/r)u_{1}^p(x)\leq  0.
\end{equation}
By Proposition \ref{strong-max-principle}, $w_{1}(x)< 0$ for every $x\in B_1(0^n)$.
The above arguments provides a starting point.
We define
\begin{equation}\lambda_0\equiv\inf\{\lambda>0 \mid w_{\mu}(x)\leq 0, \ \forall x\in  B_{\mu}(0^n),\ \ \forall \lambda\leq \mu\leq 1\}.\end{equation}
We will apply the narrow region principle (Theorem \ref{narrow-region-principle}) to show  $\lambda_0=0$. We argue by contradiction. That is, suppose $\lambda_0>0$.
Given any $0<\lambda\leq 1$, equation (\ref{equation-of-w}) gives that
\begin{eqnarray}
(-\triangle)^{\alpha/2}w_{\lambda}(x)&=&Q\Big(\frac{\lambda^2}{r}\Big)(u_{\lambda}^p(x)-u^p(x))+\Big(Q\Big(\frac{\lambda^2}{r}\Big)-Q(r)\Big)u^p(x)\nonumber\\&<&Q\Big(\frac{\lambda^2}{r}\Big)(u_{\lambda}^p(x)-u^p(x))\nonumber\\
&=&
p\cdot Q\Big(\frac{\lambda^2}{r}\Big)\cdot\psi_{\lambda}(x)\cdot w_{\lambda}(x),
\end{eqnarray}
where \begin{equation}
\min\{u_{\lambda}^{p-1}(x),u^{p-1}(x)\}\leq
 \psi_{\lambda}(x)\leq\max\{u_{\lambda}^{p-1}(x),u^{p-1}(x)\}.
 \end{equation}
Notice that
\begin{eqnarray}
u_{\lambda}(x)&\equiv& \Big(\frac{\lambda}{|x|}\Big)^{n-\alpha}u(x^{\lambda})\nonumber\\
&=&\frac{1}{\lambda^{n-\alpha}}\Big(\frac{\lambda^2}{|x|}\Big)^{n-\alpha}u\Big(\frac{\lambda^2x}{|x|^2}\Big).\end{eqnarray}
By the decay condition (\ref{decay-condition}), the above equality implies that
the function
\begin{equation}\|\psi_{\lambda_0}\|_{L^{\infty}(\dR^n)}\leq\frac{C(n,\|\tilde{u}\|_{L^{\infty}(S^n)})}{\lambda_0^{n-\alpha}}<\infty.\label{bounded-psi}\end{equation}
Therefore,
\begin{equation}
(-\triangle)^{\alpha/2} w_{\lambda_0}(x)+c_{\lambda_0}(x)\cdot w_{\lambda_0}(x)<0,\label{inequality-of-w}
\end{equation}
where
$c_{\lambda_0}(x)\equiv-p\cdot Q\Big(\frac{\lambda_0^2}{r}\Big)\cdot\psi_{\lambda_0}(x)$ is uniformly
bounded due to (\ref{bounded-psi}).

By the definition of $\lambda_0$,
\begin{equation}
w_{\lambda_0}(x)\leq 0,\ \forall x\in B_{\lambda_0}(0^n).
\end{equation}
Observe that $w_{\lambda_0}(x)<0$
for every $x\in B_{\lambda_0}(0^n)$. Otherwise, there exists $x_0\in B_{\lambda_0}(0^n)$ such that \begin{equation}w_{\lambda_0}(x_0)=\sup\limits_{x\in B_{\lambda_0}(0^n)}w_{\lambda_0}(x)=0.\end{equation} Moreover, the function $w_{\lambda_0}$ satisfies the anti-symmetry property in the sense of Proposition
\ref{strong-max-principle}.
Applying Proposition \ref{strong-max-principle}, $w_{\lambda_0}(x)\equiv0$, for all $x$ in $B_{\lambda_0}(0^n)$, which contradicts  (\ref{inequality-of-w}).

Let $\epsilon_0>0$
be the constant in Proposition \ref{narrow-region-principle},
and let $\delta_0\equiv\sup\limits_{x\in B_{\lambda_0-\frac{\epsilon_0}{10}}(0^n)}w_{\lambda_0}(x)<0$.
Then there exists $\delta_1\in(0,\min\{\lambda_0/100,\epsilon_0/100\})$ such that for every $0\leq\delta<\delta_1$, by the continuity of $w_{\lambda}$, it holds that
\begin{equation}
w_{\lambda_0-\delta}(x)<0,\ \forall x\in B_{\lambda_0-\frac{\epsilon_0}{10}}(0^n).
\end{equation}
Moreover, \begin{equation}
(-\triangle)^{\alpha/2} w_{\lambda_0-\delta}(x)+c_{\lambda_0-\delta}(x)\cdot w_{\lambda_0-\delta}(x)<0,\label{inequality-of-w-delta}
\end{equation}
and the function $c_{\lambda_0-\delta}$
is uniformly bounded provided by the choice of $\delta>0$.
Now applying narrow region principle (Theorem \ref{narrow-region-principle}), we have
$$w_{\lambda_0-\delta}(x)<0, \; \text{in}  B_{\lambda_0-\delta}(0^n), \; \forall \delta\in (0,\delta_1).$$ This
contradicts the definition of $\lambda_0$, and thus completes the proof of $\lambda_0=0$.

By the definition of $\lambda_0$,
there exists a sequence $\lambda_j\rightarrow0$ such that for every $j\in\dZ_+$,
\begin{equation}
\frac{1}{\lambda_j^{n-\alpha}}\Big(\frac{\lambda_j^2}{|x|}\Big)^{n-\alpha}u\Big(\frac{\lambda_j^2\cdot x}{|x|^2}\Big)<u(x),\ \forall x\in B_{\lambda_j}(0^n).
\end{equation}
By the decay condition (\ref{decay-condition}),
\begin{equation}
u(0)>\frac{M_0}{\lambda_j^{n-\alpha}},
\end{equation}
 and hence,
 \begin{equation}
 u(0)>\lim\limits_{j\rightarrow\infty}\frac{M_0}{\lambda_j^{n-\alpha}}=+\infty.
 \end{equation}
Contradiction.

\end{proof}

In fact, by applying a decay estimate in \cite{CFY}, Theorem \ref{prescribing-Q} can be strengthened by removing the decay condition on $u$. That is,
\begin{corollary}

Let $Q$ be continuous and rotationally symmetric on the Eucliean space $\dR^n$. In addition, assume that $Q$ is monotone in the region where $Q>0$ and $Q\not\equiv C$. Then for every $0<\alpha\leq 2$, the the equation
$$
(-\triangle)^{\alpha/2} u(x)=Q(x)\cdot u^{p}(x),\ p\equiv\frac{n+\alpha}{n-\alpha},\ x\in\dR^n,
$$
 does not admit any positive solution $u\in L_{\alpha}\cap C^{1,1}(\dR^n)$.

\end{corollary}

\section{A Liouville Type Theorem on a Half Space}

In this section, we present another application of the method of moving spheres. Specifically, we will give an alternative proof of the Liouville type theorem for the non-negative solutions to the problem
\begin{equation}
\begin{cases}
(-\triangle)^{\alpha/2}u(x)=u^p(x),\ x\in\dR_+^n,\ 1<p\leq\frac{n+\alpha}{n-\alpha}\\
u(x)=0, \ x\not\in\dR_+^n,
\end{cases}\label{upper-half-space}
\end{equation}
in both subcritical and critical cases.
This result was first proved in \cite{CFY} by an integral equation method.
Our proof here is quite different and much simpler--a combination
 of direct methods of moving spheres and moving planes.
The following narrow region principle for moving planes in \cite{CLL} will be use in our proof (see more details in that paper).

\begin{proposition}
[Narrow Region Principle for Moving Planes]\label{p:narrow-region-hyperplane} Given $\lambda\in \dR^1$, let
\begin{equation}
T_{\lambda}\equiv\{x=(x',x_n)\in\dR^n \mid x_n=\lambda\}.
\end{equation}
Denote
$$\tilde{x}^{\lambda}\equiv(x',2\lambda-x_n), \;\; \tilde{u}_{\lambda}(x)\equiv u(\tilde{x}^{\lambda}),$$
$$H_{\lambda}\equiv
\{x\in\dR^n \mid 0< x_n<\lambda\}, \; \text{and}  \;\; \tilde{H}_{\lambda}\equiv\{x\in\dR^n \mid \tilde{x}^{\lambda} \in H\}.$$
Let $w\in L_{\alpha}\cap C_{loc}^{1,1}(\Omega)$ be lower semicontinuous in $\bar{\Omega}$ such that
\begin{equation}
\begin{cases}
(-\triangle)^{\alpha/2}w(x)+c(x)w(x)\geq 0,\ x\in\Omega,\nonumber\\
w(x)\geq 0,\ x\in H_{\lambda}\setminus\Omega,
\nonumber\\
w_{\lambda}\tilde{x})=-w(x),\ x\in H_{\lambda},
\end{cases}
\end{equation}
then there exists sufficiently small
$\delta>0$
such that
if $\Omega\subset (H_{\lambda}\setminus H_{\lambda-\delta})$
 and is bounded
then $$w(x)\geq 0, \;
\forall \, x\in\Omega.$$
Furthermore, if $w(x)=0$ for some $x\in\Omega$, then
$$w(x)=0, \; \forall \, x\in\dR^n.$$

These conclusions hold for unbounded region $\Omega$ if we further assume that
$$\underset{|x| \ra \infty}{\underline{\lim}} w(x) \geq 0 .$$
\end{proposition}

\begin{theorem}\label{t:half-space}
If $u\in L_{\alpha}(\dR^n)\cap C_{loc}^{1,1}(\dR^n)$ is a nonnegative solution to (\ref{upper-half-space}), then $u\equiv0$ in $\dR^n$.
\end{theorem}

\begin{proof}
First, by the strong maximum principle, either $u(x)>0$ for every $x\in\dR_+^{n}$ or $u(x)\equiv0$ for every $x\in\dR^n$. So without loss of generality, from now on, we assume that the solution $u$ is strictly positive everywhere.

 Given $x_0\in\dR^{n-1}\times\{x_n=0\}$, define the Kelvin transform of the function $u$ as
\begin{equation}
u_{\lambda}(x)\equiv\Big(\frac{\lambda}{|x-x_0|}\Big)^{n-\alpha}\cdot u(x^{\lambda}), \;\;\;
x^{\lambda}\equiv\frac{\lambda^2 (x-x_0)}{|x-x_0|^2}+x_0.
\end{equation}
By definition, if the center $x_0$ is chosen on the boundary $\dR^{n-1}\times\{x_n=0\}$, then $x\in\dR_+^n$ implies that $x^{\lambda}\in\dR_+^n$.
Let $w_{\lambda}(x)=u_{\lambda}(x)-u(x)$, then it is straightforward that,
\begin{equation}
(-\triangle)^{\alpha/2}w_{\lambda}(x)=\Big(\frac{\lambda}{|x-x_0|}\Big)^{\tau}\cdot u_{\lambda}^p(x)-u^p(x), \ \forall x\in\dR_+^{n},
\end{equation}
where $\tau\equiv n+\alpha-p(n-\alpha)\geq 0$.

To apply the method of moving spheres, we take any $x\in B_{\lambda}(x_0)$. Denote
\begin{equation}
\Sigma_{\lambda}\equiv\{x\in B_{\lambda}(x_0)\mid w_{\lambda}(x)<0\}.
\end{equation}

{\textit{Step 1.}} We will show that exists $\delta_0>0$
such that for every $0<\lambda<\delta_0$, it holds that
\begin{equation}
w_{\lambda}(x)\geq 0,\ \forall x\in B_{\lambda}(x_0).\end{equation}
The proof of \textit{Step 1} is identical to that in the proof of Lemma \ref{l:key-lemma}.
Actually, in this case $g(t)/t^p =t^{-\frac{\tau}{n-\alpha}}$ with $\tau\equiv n+\alpha-p(n-\alpha)\geq 0$, is non-increasing. The computations are exactly the same as that in Lemma \ref{l:key-lemma}.

\vspace{0.5cm}

{\textit{Step 2.}} We will show that the solution $u$ depends only on $x_n$.

To this end, we analyze  the critical scale for the moving spheres, as that in the proof of Lemma \ref{l:key-lemma}.
For a given center $x_0\in\dR^n\cap\{x_n=0\}$, the critical scale  $\lambda_0\in(0,+\infty]$ is defined as follows,
\begin{equation}
\lambda_0\equiv\sup\{\lambda>0|w_{\mu}(x)\geq 0,\ \forall \ x\in B_{\mu}(x_0), \ \forall 0<\mu\leq\lambda\}.\label{critical-scale-2}
\end{equation}

As in the proof of \textit{Step 2} in Lemma \ref{l:key-lemma}, we have similar conclusions:

Given $x_0\in\dR^n\cap\{x_n=0\}$, the critical scale $\lambda_0>0$ is defined in (\ref{critical-scale-2}), then exactly one of the following holds:

$(a)$ for every $x_0\in\dR^n$, the corresponding critical scale $\lambda_0>0$ is finite, or

$(b)$  for every $x_0\in\dR^n$, the corresponding critical scale $\lambda_0=+\infty$, i.e, for every $\lambda>0$
\begin{equation}
w_{\lambda}(x)\geq 0, \forall x\in B_{\lambda}(x_0).
\end{equation}

If $(a)$ holds, then the decay property  \begin{equation}\lim\limits_{|x|\to+\infty}|x|^{n-\alpha}\cdot
u(x)=c_0>0\label{decay-ms}\end{equation}
follows.

Now we carry out the method of moving
planes to deduce a contradiction.
Since $u(x)>0$ for every $x\in\dR_+^n$ and $u(x)\equiv0$ for every $x\not\in\dR_+^n$, we have that $\tilde{w}_{\lambda}(x)\equiv\tilde{u}_{\lambda}(x)-u(x)\geq 0$ for every $x\in H_{\lambda}$ and for every $\lambda > 0$.
Then by proposition \ref{p:narrow-region-hyperplane}, there exists $\delta_0>0$ such that for every $0<\lambda<\delta_0$
we have that $\tilde{w}_{\lambda}(x)\geq 0$ for every $x\in H_{\lambda}$.
Define
\begin{equation}
\lambda_1\equiv\sup\{\lambda>0 \mid \tilde{w}_{\mu}(x)\geq 0,\ \forall x\in H_{\mu}, \forall \mu\leq \lambda \}.
\end{equation}
If $0<\lambda_1<\infty$, then for every $x\in H_{\lambda_1}$, by the strong maximum principle, we have
 \begin{equation}
 \tilde{u}^{\lambda_1}(x)=u(x).\label{reflection-identity}
 \end{equation}
But $u(x)=0$ for every $x\not\in\dR_+^n$, and thus the equation (\ref{reflection-identity}) contradicts to the positivity
of $u$. Therefore, $\lambda_1=+\infty$, which implies that $u$ is non-decreasing about $x_n$, which
contradicts the decay of $u$ (\ref{decay-ms}).

The contradictions in the above two possibilities of $\lambda_1$ show that
the case $(a)$ does not hold.

The above arguments show that the positive solution $u$ has to satisfy $(b)$.
We track back to the proof of proposition \ref{p:classification}.
It follows that
for every $y\in \partial\dR^n$,
\begin{equation}
\frac{(n-\alpha)u(z)+2\langle\nabla u(z),(z-y)\rangle}{|y|}\geq 0.
\end{equation}
Let $|y|\to\infty$, then we have that
\begin{equation}
\langle\nabla u(z),\nu\rangle\geq 0,
\end{equation}
where $\nu\equiv\frac{z-y}{|z-y|}$.
Since $y$ is arbitrary, the vector $\nabla u(z)$ has to be perpendicular to the hyperplane $x_n=0$. That is, the function $u$ depends only on $x_n$.

\vspace{0.5cm}

By \textit{Step 2},  the positive solution $u$ depends only on $x_n$. The arguments in section 5.3 of  the paper \cite{CFY}
shows that $u$ has to be identically zero, which gives the desired contradiction.

\end{proof}

\end{document}